\documentclass[12pt,reqno]{amsart}

\usepackage{amscd,amssymb,amsmath,amsthm}
\usepackage{graphicx}
\usepackage{geometry}
\usepackage{color}
\usepackage{comment}

\newtheorem{thm}{Theorem}
\newtheorem{defn}{Definition}

\newtheorem{cor}{Corollary}

\numberwithin{equation}{section} 

\begin{document}

\title[Gibbs measure]{Gibbs measures on spatial systems on vertices of Cayley trees}

\author{F. H. Haydarov}
 \address{F.H.Haydarov}

\address{New Uzbekistan University,
54, Mustaqillik Ave., Tashkent, 100007, Uzbekistan,}

\address{Central Asian University, 264, Milliy Bog street,  Tashkent 111221, Uzbekistan,}

 \address{Institute of mathematics,
9, University str. Tashkent, 100174, Uzbekistan,}

\address{National University of Uzbekistan,
 University str 4 Olmazor district, Tashkent, 100174, Uzbekistan.}
 \email {haydarov\_ imc@mail.ru.}

\begin{abstract}
There are many research works devoted to Gibbs measure for models on Cayley trees. Among these works, there are some works in which the general results are identical, but the considered models are various. In this article, we present the construction of Gibbs measures in the language of measure theory and reply to the question ``When can we construct Gibbs specifications?" Also, we present a new condition (convenient for verification) which is equivalent to the consistency condition of kernels. The obtained results in the article are general, not for a particular model. On the contrary, these results hold for some considered models on Cayley trees.
 \end{abstract}
\maketitle


{\bf{Key words.}} Cayley tree, spatial systems, kernels,  a version of the conditional expectation, Gibbsian specifications, Ising model.

\section{Introduction}

 The Gibbs measure is a probability measure that describes the equilibrium state of a physical system on the Cayley tree. In statistical mechanics, the Gibbs measure is used to describe the equilibrium properties of a system with many interacting particles. It assigns probabilities to different configurations of the system based on their energy. The most probable configurations are those with low energy, and the Gibbs measure gives a way to calculate these probabilities (e.g. \cite{2,g}).

Spin models on a graph or in a continuous spaces form a large
class of systems considered in statistical mechanics. Some of them
have a real physical meaning, others have been proposed as
suitably simplified models of more complicated systems. The
geometric structure of the graph or a physical space plays an
important role in such investigations. For example, in order to
study the phase transition problem on a cubic lattice $Z^d$ or in
space $\mathbb{d}$ one uses, essentially, the Pirogov-Sinai
theory; see \cite{PS1}, \cite{PS2} \cite{12}. A general methodology
of phase transitions in $\mathbb{Z}^d$ or $\mathbb{R}^d$ was
developed in \cite{7}; some recent results in this direction have
been established in \cite{MSS} (see also the
bibliography therein).

On a Cayley tree, the interactions between particles are often described by spin models, where each vertex of the tree is associated with a spin variable that can take on a certain number of values. The spins interact with their neighbors according to some interaction rule, which determines the energy of a given configuration (see \cite{9,10}).

There are many works devoted to several models defined on Cayley trees. But most of these are written by physicists with little regard for mathematical rigor. Minority of these papers which contain mathematically rigorous results about Gibbs measures on Cayley trees. This paper is mainly based on these mathematical papers. In this paper, we present the construction of Gibbsian specifications in the language of measure theory and reply to the question ``When can we construct Gibbs specifications?" Also, we present a new condition (convenient for verification) which is equivalent to the consistency condition of kernels. The obtained results in the article are general, not for a particular model. On the contrary, these results hold for some considered models on Cayley trees.

\section{Spatial systems on Cayley trees}

The Cayley tree $\Im^{k}=(V, L)$ of order $k \geq 1$ is an infinite tree, i.e. graph without cycles, each vertex of which has exactly $k+1$ edges. Here $V$ is the set of vertices of $\Im^{k}$ va $L$ is the set of its edges.


 Consider models where the spin takes values in the set $\Phi\subseteq \mathbb{R}_{\infty}^{+}$,
and is assigned to the vertices of the tree. Let $\Omega_A=\Phi^A$ be the set of all configurations
on $A$ and $\Omega:=\Phi^V$. A partial order $\preceq$ on $\Omega$ defined pointwise by stipulating that $\sigma_1 \preceq \sigma_2$ if and only if $\sigma_1(x) \leq \sigma_2(x)$ for all $x \in V$. Thus $(\Omega, \preceq)$ is a poset, and whenever we consider $\Omega$ as a poset then it will always be with respect to this partial order. The poset $\Omega$ is complete. If $\Omega_1$ is a non-empty subset of $\Omega$ then its least upper bound is the mapping given by $\sup (\Omega_1)(x)=\sup \{\sigma(x): \sigma \in \Omega_1\}$ for all $x \in V$.

We consider all elements of $V$ are numerated (in any order) by the numbers: $0,1,2,3,...$. Namely, we can write $V=\{x_0, x_1, x_2, ....\}$.

We denote by $\mathcal{N}$ the set of all finite subsets of $V$. For each $A\in V$ let $\pi_A: \Omega \rightarrow\Phi^{A}$ be given by $\pi_A\left(\sigma_x\right)_{x\in V}=\left(\sigma_x\right)_{x\in A}$ and let $\mathcal{C}_A=\pi_A^{-1}\left(\mathcal{P}\left(\Phi^{A}\right)\right)$. Let $\mathcal{C}=\bigcup_{A \in \mathcal{N}} \mathcal{C}_A$ and $\mathcal{F}$ is the smallest sigma field containing $\mathcal{C}$. 

For a fixed $x^{0} \in V$ we put
$$
W_{n}=\left\{x \in V \mid d\left(x, x^{0}\right)=n\right\}, \quad V_{n}=\bigcup_{m=0} ^{n} W_{m}, $$

where $d(x, y)$ is the distance between the vertices $x$ and $y$ on the Cayley tree, i.e. the number of edges of the shortest walk (i.e., path) connecting vertices $x$ and $y$.

For any fixed configuration $\sigma_{A}\in\Omega_{A}, \ A\subset V$ we denote:
$$\bar{\sigma}_{A}:=\left\{\sigma\in\Omega :\;
\sigma\big|_{A}=\sigma_A\right\}.$$

\begin{cor}\label{muhim}\cite{part1}
$\mathcal{F}=\mathcal{S}\left(\left\{\bar{\sigma}_{V_n}: n\in \mathbb{N}_{0}\right\}\right)$.
\end{cor}

A measurable space $(\Omega, \mathcal{F})$, a non-empty index set $V$ equipped with a partial order $\subseteq$, and a decreasing family $\mathbb{F}=\left\{\mathcal{F}^\Lambda\right\}_{\Lambda \in \mathcal{N}}$, (where $\mathcal{F}^\Lambda:=\mathcal{S}(\mathcal{C}_\Lambda)$) of sub $\sigma$-fields of $\mathcal{F}$. It is easy to check that the poset $(\mathcal{N}, \subseteq)$ is directed (i.e., for all $A_1, A_2 \in D$ there exists $A \in D$ with $A_1 \preceq A$ and $A_2 \preceq A$ ) and countably generated (i.e., there exists a countable subset $D_0$ of $D$ such that for each $A \in D$ there is an element $A_0 \in D_0$ with $A \preceq A_0$ ).

A collection $\Sigma:=\left(V, \mathcal{N}, \Omega,\left\{\mathcal{F}^{\Lambda}\right\}_{\Lambda \in \mathcal{N}}\right)$ with $(V, \mathcal{N}),$ $\Omega$ and $\left\{\mathcal{F}^{\Lambda}\right\}_{\Lambda \in \mathcal{N}}$ as above and with the inclusion order $\subseteq$ on $\mathcal{N}$ countably generated will be called a \emph{\textbf{spatial system}}. For instance, the lattice and particle models are both spatial systems. Note that we choose the set of spin values $\Phi$ which is $\mathcal{N}$ is countably generated (see in \cite{part1}).

For each $\Lambda \in \mathcal{N}$ let $\mu_{\Lambda}$ be a probability measure on $\mathcal{F}^{\Lambda}$. Then the family of probability measures $\left\{\mu_{\Lambda}\right\}_{\Lambda \in \mathcal{N}}$ is said to be \emph{\textbf{consistent (compatible)}} if $\mu_{\Lambda}(F)=\mu_{\Delta}(F)$ for all $F \in \mathcal{F}^{\Lambda}$ whenever $\Lambda \subset \Delta$. We say that the family of $\sigma$-algebras $\left\{\mathcal{F}^{\Lambda}\right\}_{\Lambda \in \mathcal{N}}$ has the Kolmogorov extension property if for each consistent family $\left\{\mu_{\Lambda}\right\}_{\Lambda \in \mathcal{N}}$ there exists a probability measure $\mu$ on $\mathcal{F}$ such that $\mu(F)=\mu_{\Lambda}(F)$ for all $F \in \mathcal{F}^{\Lambda}, \ \Lambda \in \mathcal{N}$. If $\mu$ exists then by Carath\'{e}odory's extension theorem it is unique, since it is determined by the family $\left\{\mu_{\Lambda}\right\}_{\Lambda \in \mathcal{N}}$ on the algebra $\mathcal{C}$ and $\sigma(\mathcal{C})=\mathcal{F}$. Finally, we say that the spatial system $\Sigma$ has the Kolmogorov extension property if the family of $\sigma$-algebras $\left\{\mathcal{F}^{\Lambda}\right\}_{\Lambda \in \mathcal{N}}$ does. Let $\mathcal{N}_1$ be a subset of $\mathcal{N}$ which is also directed and countably generated (\textbf{cofinal} in \cite{g}).

 The following theorem gives us any consistent family of probability measures $\left\{\mu_{\Lambda}\right\}_{\Lambda \in \mathcal{N}_1}$ extends to a probability measure $\mu \in \mathrm{P}(\mathcal{F})$.
 \begin{thm}\label{thm2.1} Let $\mathcal{N}_1$ be a directed and for each $\Lambda \in \mathcal{N}_1$ let $\mu_{\Lambda} \in \mathrm{P}\left(\mathcal{F}^{\Lambda}\right)$. If $\left\{\mu_{\Lambda}\right\}_{\Lambda \in \mathcal{N}_1}$ is consistent then
 there exists a unique probability measure $\mu \in \mathrm{P}(\mathcal{F})$ such that $\mu(F)=\mu_{\Lambda}(F)$ for all $F \in \mathcal{F}^{\Lambda}, \Lambda \in \mathcal{N}_1$.
\end{thm}

\begin{proof} Let $\mathcal{N}_1=\{\Lambda_1, \Lambda_2, ..., \Lambda_s, ...\}$, where $|\Lambda_i|<\infty$ for all $i\in\mathbb{N}$. Since $\mathcal{N}_1$ is directed, there exists $\Lambda'_i \in \mathcal{N}_1$ with $\Lambda'_{i-1} \subseteq \Lambda'_{i}$ and $\Lambda_i \subseteq \Lambda'_i$ where $\Lambda'_{1}=\Lambda_{1}$. Thus, this gives us  $\Lambda'_{i}\subseteq \Lambda'_{i+1}$ for all $i\in \mathbb{N}$. Let $\mathcal{N}'_1:=\{\Lambda'_1, \Lambda'_2, ..., \Lambda'_s, ...\}$ then for each $\Lambda \in \mathcal{N}$ there is an element $\Lambda'_q \in \mathcal{N}'_1$ with $\Lambda \subseteq \Lambda'_q$. Hence, by Kolmogorov extension theorem there exists a unique probability measure $\mu' \in \mathrm{P}(\mathcal{F}')$ such that $\mu'(F)=\mu'_{\Lambda'_i}(F)$ for all $F \in \mathcal{F}^{\Lambda'_i}, \Lambda'_i \in \mathcal{N}'_1$, where $\mathcal{F}'$ is a $\sigma$-field generated by $\bigcup_{i=1}^{\infty}\mathcal{F}^{\Lambda'_i}$. Since $\left\{\mathcal{F}^\Lambda\right\}_{\Lambda \in \mathcal{N}_1}$ is a decreasing family we get $\mathcal{F}=\mathcal{F}'$. By Carath\'{e}odory's extension theorem we get $\mu'=\mu$.
\end{proof}

\section{Gibbsian specifications on Cayley trees}

Note that $\Sigma=\left(V, \mathcal{N}, \Omega,\left\{\mathcal{F}^{\Lambda}\right\}_{\Lambda \in \mathcal{N}}\right)$ is a spatial system which defined in previous section.  $\mathcal{N}$ equipped with a directed, countably generated partial order $\subseteq$, and a decreasing family $\mathbb{F}=\left\{\mathcal{F}_\Lambda\right\}_{\Lambda \in \mathcal{N}}$ of sub-$\sigma$-algebras of $\mathcal{F}$.

Suppose for each $\Lambda \in \mathcal{N}$ we have a strict $\mathcal{F}_\Lambda$-measurable quasi-probability kernel $\zeta_\Lambda \in K(\mathcal{F}_\Lambda)$. Then the family $\mathcal{V}=\left\{\zeta_\Lambda\right\}_{\Lambda \in \mathcal{N}}$ will be called an $\mathbb{F}$-\emph{\textbf{specification}} if $\zeta_{\Delta}=\zeta_{\Delta} \zeta_{\Lambda}$ whenever $\Lambda, \Delta \in \mathcal{N}$ with $\Lambda \subseteq \Delta$.  Let $\mathcal{V}=\left\{\zeta_\Lambda\right\}_{\Lambda \in \mathcal{N}}$ be an $\mathbb{F}$-specification; then a probability measure $\mu \in \mathrm{P}(\mathcal{F})$ is called a \emph{\textbf{Gibbs state with specification $\mathcal{V}$}} if $\mu=\mu \zeta_\Lambda$ for each $\Lambda \in \mathcal{N}$. Note that this definition of Gibbs states originates from Dobrushin \cite{ch1, ch2, ch3, ch4}, and Lanford and Ruelle \cite{ch12, ch15}.

\begin{defn}\label{Definition 6.14.}  Let $P_\Lambda: \Omega \rightarrow \overline{\mathbb{R}}:=\mathbb{R}\cup\{-\infty, \infty\}$ be $\mathcal{F}_\Lambda$-measurable mapping for all $\Lambda \in \mathcal{N}$, then the collection $P=\left\{P_\Lambda\right\}_{\Lambda\in \mathcal{N}}$ is called \textbf{a potential}. Also, the following expression
\begin{equation}
H_{\Delta, P}(\sigma) \stackrel{\text { def }}{=} \sum_{\Delta \cap \Lambda \neq \varnothing, \Lambda \in \mathcal{N}} P_\Lambda(\sigma), \quad \forall \sigma \in \Omega.
\end{equation}
is called \textbf{Hamiltonian} $H$ associated to the potential $P$.
\end{defn}
Put
$$
r(P) \stackrel{\text { def }}{=} \inf \left\{R>0: P_\Lambda \equiv 0 \text { for all } \Lambda \text { with } \operatorname{diam}(\Lambda)>R\right\}.
$$
If $r(P)<\infty, P$ has finite range and $H_{\Delta ; P}$ is well defined. If $r(P)=\infty$, $P$ has infinite range and, for the Hamiltonian to be well defined, we will assume that $P$ is absolutely summable in the sense that
$$
\sum_{\substack{\Lambda \in \mathcal{N}, x\in \Lambda }}\left\|P_\Lambda\right\|_{\infty}<\infty, \quad \forall x\in V,
$$
(remember that $\left.\|f\|_{\infty} \stackrel{\text { def }}{=} \sup_\omega|f(\omega)|\right)$ which ensures that the interaction of a spin with the rest of the system is always bounded, and therefore that $\left\|H_{\Delta ; P}\right\|_{\infty}<\infty$.

Let $\mathcal{N}_1=\{V_n : n\in \mathbb{N}\}$ then we define the following Hamiltonian in the box $V_n, n\in\mathbb{N}$:
\begin{equation}\label{H}
H_n(\sigma)=\sum_{n} P_{V_n}(\sigma), \quad \forall \sigma \in \Omega.
\end{equation}

Let us define a specification $\zeta^{H}=\left\{\zeta_{V_n}^{H}\right\}_{n\in\mathbb{N}}$ (in short $\zeta_{V_n}^{H}:=\zeta_{n}^{H}$) such that $\zeta_{V_n}^{H}(\cdot \mid \omega)$ gives to each configuration $\tau_{V_n} \omega_{V_n}^c$ a probability proportional to the Boltzmann weight prescribed by equilibrium statistical mechanics (see e.g. \cite{2}):
\begin{equation}\label{teng1}
\zeta_{n}^{P}\left(\omega, \sigma_{n}\right) \stackrel{\text { def }}{=} \frac{1}{\mathbf{Z}_{n}^{\omega}} e^{-H_n\left(\sigma_{n} \omega_{\bar{V}_n}\right)},
\end{equation}
where we have written explicitly the dependence on $\omega_{\bar{V}_n}$, and $\mathbf{Z}_{n}^\omega$ is a partition function, i.e.,
$$
\mathbf{Z}_{n}^{\omega} \stackrel{\text { def }}{=} \sum_{\sigma_{n} \in \Omega_{V_n}} \exp \left(-H_n\left(\sigma_{n} \omega_{\bar{V}_n}\right)\right).
$$
\begin{thm}\label{6.15.} If $\mathbb{F}_1:=\left\{\zeta_{n}^{P}\right\}_{n\in\mathbb{N}}$ then $\mathbb{F}_1$ is a specification.\end{thm}
\begin{proof} For a fixed $m,n\in\mathbb{N}$ with $V_n \subset V_m \Subset V$, we show that $\zeta_{m}^{P} \zeta_{n}^{P}=\zeta_{m}^{P}$. Let $A\subset \mathcal{P}(\Omega_{V_n})$. At first, we consider the case $A=\{\sigma_{V_n}\}$, i.e.,
$$
\begin{aligned}
\zeta_{m} \zeta_{n}(\omega, \sigma_{V_n}) & =\sum_{\tau_{V_m}} \zeta_{m}\left(\omega, \tau_{V_m} \right) \zeta_{n}\left(\tau_{V_m} \omega_{V_m^c}, \sigma_{V_n}\right)=\sum_{\tau_{V_m}} \zeta_{m}\left(\omega, \tau_{V_m} \right) \zeta_{n}\left(\tau_{V_m \backslash V_n} \omega_{V_m^c}, \sigma_{V_n}\right).
\end{aligned}
$$
From the last equation, for any $A\subset \mathcal{P}(\Omega_{V_n})$ we obtain
$$
\begin{aligned}
\zeta_{m} \zeta_{n}(A \mid \omega)=\sum_{\tau_{V_m}} \sum_{\eta_{\Delta}} \mathbf{1}_A\left(\eta_{V_n} \tau_{V_m \backslash V_n} \omega_{V_m^c}\right) \zeta_{m}\left(\omega, \tau_{V_m}\right) \zeta_{n}\left(\tau_{V_m \backslash V_n} \omega_{V_m^c}, \eta_{V_n}\right).
\end{aligned}
$$

For any configuration $\tau_{V_m}$ which defined on $V_m$ we rewrite this configuration as a combination of $\tau_{V_n}^{\prime}$ on $V_n$ and $\tau_{V_m \backslash V_n}$ on $V_m \backslash V_n$, i.e. $\tau_{V_m}=\tau_{V_n}^{\prime} \vee\tau_{V_m \backslash V_n}^{\prime \prime}$. By (\ref{teng1}), RHS of the last equation can be rewritten as:
$$
\sum_{\tau_{V_m \backslash V_n}^{\prime \prime}} \sum_{\eta_{V_n}} \mathbf{1}_A\left(\eta_{V_n} \tau^{\prime \prime}_{V_m \backslash V_n} \omega_{V_m^c}\right) \frac{e^{-H_{V_n}\left(\eta_{V_n} \tau^{\prime \prime}_{V_m \backslash V_n} \omega_{V_m^c}\right)}}{\mathbf{Z}_{m}\left(\omega_{V_m^c}\right) \mathbf{Z}_{n}\left(\tau_{V_m \backslash V_n}^{\prime \prime} \omega_{V_m^c}\right)} \sum_{\tau_{V_n}^{\prime}} e^{-H_{V_m}\left(\tau_{V_n}^{\prime} \tau_{V_m \backslash V_n}^{\prime \prime} \omega_{V_m^c}\right)} .
$$
It's easy to check that the following expression
$$H_{m}\left(\tau_{V_n}^{\prime} \tau_{V_m \backslash V_n}^{\prime \prime} \omega_{V_m^c}\right)- H_{n}\left(\tau_{V_n}^{\prime} \tau_{V_m \backslash V_n}^{\prime \prime} \omega_{V_m^{\mathrm{c}}}\right)$$
does not depend on $\tau_{V_n}^{\prime}$. That's why we have
$$H_{m}\left(\tau_{V_n}^{\prime} \tau_{V_m \backslash V_n}^{\prime \prime} \omega_{V_m^c}\right)-H_{n}\left(\tau_{V_n}^{\prime} \tau_{V_m \backslash V_n}^{\prime \prime} \omega_{V_m^{\mathrm{c}}}\right)=H_{m}\left(\eta_{V_n} \tau_{V_m \backslash V_n^{\prime \prime}} \omega_{V_m^{\mathrm{c}}}\right)-H_{n}\left(\eta_{V_n} \tau_{V_m \backslash V_n^{\prime \prime}} \omega_{V_m^{\mathrm{c}}}\right),
$$
which gives
$$
\sum_{\tau_{V_n}^{\prime}} e^{-H_{m}\left(\tau_{V_n}^{\prime} \tau_{V_m \backslash V_n}^{\prime \prime} \omega_{V_m^c}\right)}=\mathbf{Z}_{n}\left(\tau^{\prime \prime}_{V_m \backslash V_n}\omega_{V_m^c}\right) e^{H_{m}\left(\eta_{V_n} \tau_{V_m \backslash V_n}^{\prime \prime}{ }^\omega{ }^c\right)} e^{-H_{n}\left(\eta_{V_n} \tau_{V_m \backslash V_n}^{\prime \prime} \omega_{V_m^c}\right)} .
$$
Inserting this in the above expression, and renaming $\eta_{V_n} \tau_{V_m \backslash V_n}^{\prime \prime} \equiv \eta_{V_m}^{\prime}$, we get
$$
\zeta_{m} \zeta_{n}(A \mid \omega)=\sum_{\eta_{V_m}^{\prime}} \mathbf{1}_A\left(\eta_{V_m}^{\prime} \omega_{V_m^c}\right) \frac{e^{-H_{m}\left(\eta_{V_m}^{\prime} \omega_{V_m c}\right)}}{\mathbf{Z}_{m}\left(\omega_{V_m^c}\right)}=\zeta_{m}(A \mid \omega).$$\end{proof}

Now we give more comfortable condition for special kernels which is equivalent to the condition of consistency.

We consider some models discussed in this section with the corresponding potentials on Cayley trees:
\begin{equation}\label{11}\bar{P}_B(\sigma)= \begin{cases}-\beta \rho(\sigma(x), \sigma(y)) & \text { if }  B=\langle x, y \rangle,  \\ 0 & \text { otherwise.}\end{cases}
\end{equation}
Let $h:(\Omega, \mathcal{P}(\Omega)) \rightarrow (\mathbb{R}, \mathcal{B}(\mathbb{R}))$ be a measurable mapping. We define a kernel:
\begin{equation}\label{2.6}\zeta_n^{\bar{P}}(h, \sigma_n)=\frac{1}{\mathbf{Z}^{h}_{n}}\exp\left\{ J\beta \sum_{V_n \cap \Lambda \neq \varnothing, \Lambda \in \mathcal{N}} \bar{P}_\Lambda(\sigma_n)+\sum_{x\in W_n}\rho(\sigma_n(x), h(x))\right\},\end{equation}
where
$$\mathbf{Z}^{h}_{n}=\sum_{\sigma_n\in \Omega_{V_n}}\exp\left\{ J\beta \sum_{V_n \cap \Lambda \neq \varnothing, \Lambda \in \mathcal{N}} \bar{P}_\Lambda(\sigma_n)+\sum_{x\in W_n}\rho(\sigma_n(x),h(x))\right\}.$$
It's easy to check that $\zeta_n^{\bar{P}}(\sigma_n, h)$ satisfies the conditions of kernel.

 Are the family of kernels $\{\zeta_n^{\bar{P}}(\sigma_n, h)\}_{n\in\mathbb{N}}$ specification or not?  For a fixed $m,n\in\mathbb{N}$ with $V_n \subset V_m \Subset V$, we show that $\zeta_{m}^{\bar{P}} \zeta_{n}^{\bar{P}}=\zeta_{m}^{\bar{P}}$. But it is easy to check the condition. From Corollary \ref{muhim} it is enough to check the condition $\zeta_{n+1}^{\bar{P}} \zeta_{n}^{\bar{P}}=\zeta_{n+1}^{\bar{P}}$. Also, the last equality is equivalent to the following equality:
 \begin{equation}\label{2.7} \sum_{\omega_n\in \{-1, 1\}^{W_n}}\zeta_n^{\bar{P}}(\sigma_{n-1}\omega_n, h)=\zeta_{n-1}^{\bar{P}}(\sigma_{n-1}, h).
 \end{equation}
 We show that more comfortable condition which is equivalent to (\ref{2.7}).

\begin{thm}\label{THEOREM 2.1.} Probability distributions $\zeta_n^{\bar{P}}(h, \sigma_n), n=1,2, \ldots$, in (\ref{2.6}) are compatible if and only if for any $x \in V$ the following equation holds:
\begin{equation}\label{2.8} \prod_{y \in S(x)} \frac{\sum_{u \in\{-1,1\}} \exp \left(J \beta \rho(1, u)+\rho(u, h_y)\right)}{\sum_{u \in\{-1,1\}} \exp \left(J \beta \rho(-1, u)+\rho(u, h_y) \right)}=\exp \{\rho(1, h_x)-\rho(-1, h_x)\}.\end{equation}
\end{thm}
\begin{proof} \textit{Necessity.} Suppose that (\ref{2.7}) holds; we want to prove (\ref{2.8}). Substituting (\ref{2.6}) in (\ref{2.7}), obtain that for any configurations $\sigma_{n-1}: x \in V_{n-1} \mapsto \sigma_{n-1}(x) \in\{-1,1\}$ :

From (\ref{11}) we can rewrite (\ref{2.6}) as
\begin{equation}\label{1111}\zeta_n^{P}(h, \sigma_n)=\frac{1}{\mathbf{Z}^{h}_{n}}\exp\left\{J\beta \sum_{\langle x,y \rangle, x,y\in V_n} \rho(\sigma_n(x), \sigma_n(y))+\sum_{x\in W_n}\rho(\sigma_n(x), h(x))\right\},\end{equation}
\begin{equation}\label{2.9}
\begin{aligned}
& \frac{\mathbf{Z}^{h}_{n-1}}{\mathbf{Z}^{h}_n} \sum_{\omega_n \in \Omega_{W_n}} \exp \left(\sum_{x \in W_{n-1}} \sum_{y \in S(x)}\left(J \beta \rho(\sigma_{n-1}(x), \omega_n(y))+\rho(\omega_n(y), h_y)\right)\right)= \\
& \exp \left(\sum_{x \in W_{n-1}} \rho(\sigma_{n-1}(x), h_x)\right),
\end{aligned}
\end{equation}
where $\omega_n: x \in W_n \mapsto \omega_n(x)$.
From (\ref{2.9}) we get:
$$
\begin{gathered}
\frac{\mathbf{Z}^{h}_{n-1}}{\mathbf{Z}^{h}_n} \sum_{\omega_n \in \Omega_{W_n}} \prod_{x \in W_{n-1}} \prod_{y \in S(x)} \exp \left(J \beta \rho(\sigma_{n-1}(x), \omega_n(y))+\rho(\omega_n(y),h_y)\right)= \\
\prod_{x \in W_{n-1}} \exp \left( \rho(\sigma_{n-1}(x), h_x)\right) .
\end{gathered}
$$
Rewrite now the last equality for $\sigma_{n-1}(x)=1$ and $\sigma_{n-1}(x)=-1$, then dividing first of them to the second one we get
$$\prod_{y \in S(x)} \frac{\sum_{u \in\{-1,1\}} \exp \left(J \beta \rho(1, u)+\rho(u, h_y)\right)}{\sum_{u \in\{-1,1\}} \exp \left(J \beta \rho(-1, u)+\rho(u, h_y) \right)}=\exp \{\rho(1, h_x)-\rho(-1, h_x)\}.$$

\textit{Sufficiency.} Suppose that (\ref{2.8}) holds. It is equivalent to the representations
\begin{equation}\label{2.11}
\prod_{y \in S(x)} \sum_{u \in\{-1,1\}} \exp\left(J \beta \rho(t, u)+\rho(u, h_y)\right)=a(x) \exp \left(\rho(t, h_x)\right), t \in\{-1,1\}
\end{equation}
for some function $a(x)>0, x \in V$. We have
\begin{equation}\label{2.12}
\begin{gathered}
\text {LHS of} \ (\ref{2.7})=\frac{1}{\mathbf{Z}^{h}_n} \exp \left(J\beta \sum_{\langle x,y \rangle, x,y\in V_{n-1}} \rho(\sigma_{n-1}(x), \sigma_{n-1}(y))\right) \times \\
\prod_{x \in W_{n-1}} \prod_{y \in S(x)} \sum_{u \in\{-1,1\}} \exp \left(J \beta \rho(\sigma_{n-1}(x), u)+\rho(u, h_y)\right).
\end{gathered}
\end{equation}
Substituting (\ref{2.11}) into (\ref{2.12}) and denoting $\mathbf{A}^{h}_{n-1}=\prod_{x \in W_{n-1}} a(x)$, we get
\begin{equation}\label{2.13}
\begin{gathered}
\text {RHS of} \ (\ref{2.12})=\frac{\mathbf{A}^{h}_{n-1}}{\mathbf{Z}^{h}_n} \exp \left(J\beta \sum_{\langle x,y \rangle, x,y\in V_{n-1}} \rho(\sigma_{n-1}(x), \sigma_{n-1}(y))\right) \times \\ \prod_{x \in W_{n-1}} \exp \left(\rho(\sigma_{n-1}(x), h_x)\right).
\end{gathered}
\end{equation}
Since $\zeta_n^{\bar{P}}, n \geq 1$ is a probability, we should have
$$
\sum_{\sigma_{n-1} \in \Omega_{V_{n-1}}} \sum_{\omega_n \in \Omega_{W_n}} \zeta_n^{\bar{P}}\left(\sigma_{n-1}\omega_n, h\right)=1
$$
Hence from (\ref{2.13}) we get $\mathbf{Z}^{h}_{n-1} \mathbf{A}^{h}_{n-1}=\mathbf{Z}^{h}_n$, and (\ref{2.7}) holds.
\end{proof}

\textit{Application to Ising model.}   One of classic models in statistical mechanics is Ising model. We give above results for Ising model. Given a finite volume $\Lambda$ in $\Im^{k}$, for a prescribed boundary condition (b.c.) $\omega_{\Lambda^c} \in \Omega_{\Lambda^c}=$ $\{-1,1\}^{\Lambda^c}$, we define Hamiltonians on $\Omega$ for any $\sigma \in \Omega$ to be the uniformly convergent series
$$
H_{\Lambda}^\omega\left(\sigma_{\Lambda}\right)=-\sum_{\substack{x, y \in \Lambda \\ \langle x, y\rangle}} J \sigma_x \sigma_y-\sum_{\substack{x \in \partial(\Lambda), y \in V\setminus \Lambda \\
\langle x,y\rangle}} J \sigma_x \omega_y.
$$
For a fixed inverse temperature $\beta>0$, the Gibbs specification is determined by a family of probability kernels $\zeta=\left(\zeta_{\Lambda}\right)_{\Lambda \in \mathcal{S}}$ defined on $\Omega_{\Lambda} \times \mathcal{F}_{\Lambda^c}$ by the Boltzmann-Gibbs weights
\begin{equation}\label{xino}
\zeta_{\Lambda}(\sigma \mid \omega)=\frac{1}{Z_{\Lambda}^\omega} e^{-\beta H_{\Lambda}^\omega\left(\sigma_{\Lambda}\right)}
\end{equation}
where $Z_{\Lambda}^\omega=\sum_{\sigma \in \Omega_{\Lambda}} e^{-\beta H_{\Lambda}^\omega\left(\sigma_{\Lambda}\right)}$ is the partition function, related to free energy.

 Note that by Theorem \ref{6.15.} we can conclude that the family of kernels $\{\zeta_{\Lambda}(\sigma \mid \omega)\}_{\lambda\in \mathcal{N}}$ satisfies conditions of specification. If we take $\rho{\sigma, \omega}=\sigma\cdot \omega$ and $h_x=J\sum_{y\in S(x)}\omega(y)$ then from Theorem \ref{THEOREM 2.1.} the consistency condition will be equivalent to the following euqation:
$$
\prod_{y \in S(x)} \frac{\sum_{u \in\{-1,1\}} \exp \left(J \beta u+h_y u\right)}{\sum_{u \in\{-1,1\}} \exp \left(-J \beta u+h_y u\right)}=\exp \left(2 h_x\right).
$$
If we define new function
$$
f(h, \theta)=\frac{1}{2} \ln \frac{(1+\theta) e^{2 h}+(1-\theta)}{(1-\theta) e^{2 h}+(1+\theta)} .
$$
the the last equation can be rewritten as
$$
h_x=\sum_{y \in S(x)} f\left(h_y, \theta\right).
$$

Namely, we give well-known theorem for Ising models on Cayley trees (see e.g. \cite{10}).
\begin{cor} On Cayley trees the probability kernels $\{\zeta_{\Lambda}(\sigma \mid \omega)\}_{\lambda\in \mathcal{N}}$ for Ising model are consistent iff for any $x \in V$ the following equation holds:
$$
h_x=\sum_{y \in S(x)} f\left(h_y, \theta\right).
$$
\end{cor}

Each solution to the equation $h_x=\sum_{y \in S(x)} f\left(h_y, \theta\right)$ gives us a splitting Gibbs measure. If we choose $\rho$ as a Kronecker symbol, then the obtained model is called Potts model. Our results hold for Potts model (see detail in \cite{GRHR,Gan, HA,10}).

\section*{Acknowledgements}
The work supported by the fundamental project (number: F-FA-2021-425)  of The Ministry of Innovative Development of the Republic of Uzbekistan.
\section*{Statements and Declarations}

{\bf	Conflict of interest statement:}
The author states that there is no conflict of interest.

\section*{Data availability statements}
The datasets generated during and/or analysed during the current study are available from the corresponding author on reasonable request.

\end{document}